\documentclass[12pt, a4paper]{article}
\usepackage{amsmath, amscd, amsfonts, amssymb, amstext, amsthm}
\usepackage[a4paper]{geometry}

\title{Asymptotics of eigensections on toric varieties}
\author{A.Huckleberry and H.Sebert\\
Appendix by D. Barlet}
\date{\today}  \parindent0ex 
\theoremstyle{plain} \newtheorem{theorem} {Theorem}
[section]
\newtheorem{lemma} [theorem]{Lemma}
\newtheorem{proposition}[theorem]{Proposition}
\newtheorem{corollary} [theorem]{Corollary}
\theoremstyle{definition} \newtheorem*{definition} {Definition}
\newtheorem*{example} {Example}
\newtheorem*{remark} {Remark}
\newtheorem*{remarks} {Remarks}
\newcommand{\hq } {{/\kern -.185em/}}
\newcommand{\cupdot}{\,\,\dot\cup\,\,}
\begin{document}
\maketitle \noindent

%
%
%
%
%

\begin {abstract}
\noindent Using exhaustion properties of invariant plurisubharmonic
functions along with basic combinatorial information on toric
varietes we prove convergence results for sequences of distribution
functions $\varphi _n=\vert s_N\vert / \Vert s_N\Vert_{L^2}$ for
sections $s_N\in \Gamma (X,L^N)$ approaching a semiclassical ray.
Here $X$ is a normal compact toric variety and $L$ is an ample line
bundle equipped with an arbitrary positive bundle metric which is
invariant with respect to the compact form of the torus. Our work
was motivated by and extends that of Shiffman, Tate and 
Zelditch \cite {STZ}.
\end {abstract}

\tableofcontents

\section{Notation and statement of results}\label{labSecTorusLifting}
Let us begin by describing the basic setting of this paper.
For this let $X$ be an $m$-dimensional connected normal compact
complex space equipped with an effective holomorphic action of
a complex torus $T\cong (\mathbb C^*)^m$. It follows that $T$
has a (unique, Zariski dense) open orbit $T.x_0$ where the
base point $x_0$ is fixed for the discussion. We consider a
very ample line bundle $\pi :L\to X$ to which the elements 
of $T$ can be lifted in the sense that for every $t\in T$
there is a holomorphic bundle mapping $\hat {t}:L\to L$
with $t\pi=\pi \hat {t}$.  If $\widehat T$ denotes the
group of bundle transformations which arise in this way, then
$\widehat {T}\cong (\mathbb C^*)^{m+1}$ and $\pi $ induces an
exact sequence $1\to \mathbb C^*\to \widehat {T}\to T\to 1$.  

\bigskip\noindent 
The group $\widehat {T}$ is naturally represented on the space
$\Gamma (X,L)$ of sections by $\hat {t}(s):=\hat {t}st^{-1}$.
Since $\widehat {T}\cong (\mathbb C^*)^{m+1}$, this representation is
completely reducible. Note that if $s_1,s_2\in \Gamma (X,L)$
are eigensections which transform by the same character, then
$s_1s_2^{-1}$ is a $T$-invariant meromorphic function on $X$ which, since
$T$-has an open orbit, is constant.  Hence the representation is
multiplicity-free.

\bigskip\noindent
In order to lift the $T$-action to $L$ we fix a base eigensection
$s_0\in \Gamma (X,L)$ and define the base point $1_{x_0}:=s(x_0)$ in the
fiber $L_{x_0}$ over the base point $x_0$ in the open $T$-orbit.
Now $\hat {t}(s_0)=\chi (\hat {t})s_0$ for some character
$\chi \in \mathfrak X(\widehat {T})$.  Thus $\mathrm {Ker}(\chi _0)$
is identified with $T$ by its orbit of $1_{x_0}$ which is mapped 
bijectively onto the open orbit in $X$.  We choose this lifting
of  $T$ as a group of bundle transformations, i.e.,
$T\cong \mathrm {Ker}(\chi _0)\hookrightarrow \widehat {T}$. 
Since $L$ is assumed to be very
ample and the associated embedding 
$\varphi _L:X\to \mathbb P(\Gamma (X,L)^*)$ is equivariant, it
follows from the fact that every holomorphic representation of
$T$ is algebraic that the $T$-action on $X$ is algebraic.
Consequently $X$ is a toric variety (see, e.g.,\cite{F} definitions
and basic results).

\bigskip\noindent
Now let $L$ be equipped with a smooth Hermitian bundle metric $h$
which is positive in the sense that for every local section
$s$ the function $-\log \vert s\vert_h^2$ is strictly plurisubharmonic.
A function on a complex space is said to be smooth if it
can be locally extended to a smooth function in a local embedding
space of $X$ in a complex manifold.  It is said to be strictly
plurisubharmonic if the extended function is strictly plurisubharmonic.
We also assume that $h$ is invariant with respect to the maximal
compact subgroup $T_\mathbb R$ of $T$.  This can be achieved by
averaging.  Of fundamental importance here 
is the $L_2$-norm  $\Vert s\Vert_h^2:=\int _X\vert s\vert^2_hd\lambda $.
The measure $d\lambda $ is chosen to be associated to the volume form
$\omega ^m$ of a K\"ahler metric.  The latter is defined 
on a covering $\{U_\alpha\}$ by strictly plurisubharmonic potential
functions $h_\alpha $ where the differences $h_\beta -h_\alpha $ are
pluriharmonic on the intersections $U_{\alpha\beta}$.  
Thus $\omega $ is locally the $(1,1)$-form 
$\omega _\alpha =\frac{i}{2}\partial \bar {\partial}h_\alpha$.
One checks that any two such measures are equivalent.

\subsubsection* {The weight lattice}
Having chosen $1=1_{x_0}$, in every isotypical component $V_\chi$ we
have a unique section $s$ with $s(x_0)=1$. The character $\chi $
defines $s$ on the open orbit by $t^{-1}s(t.x_0)=\chi (t)s(x_0)=\chi
(t)\cdot 1$. Since $s_0$ is $T$-invariant, it follows that
$s(t.x_0)=\chi (t)s_0(t.x_0)$. Consequently, having fixed $s_0$ all
other eigensections are determined by their characters and as a
result we turn to the space of characters on $T$.

\bigskip\noindent
Characters $\chi :T\to \mathbb C^*$
restrict to characters $\chi: T_{\mathbb R}\to S^1$ and conversely
such compact characters extend uniquely to characters of $T$. Thus
it is traditional to write a character as a compact character $\chi
=e^{2\pi i\alpha }$ where the linear function $\alpha \in \mathfrak
{t}^*_{\mathbb R}$ is required to take on integral values on the
kernel of $\mathrm {exp}:\mathfrak {t}_\mathbb R\to T_{\mathbb R}$.
Of course we implicitly also regard such linear functions as being
complex linear, i.e., in $\mathfrak {t}^*$ where they define the
complex characters $\chi $. We denote the space of such functions by
$\mathfrak {t}^*_{\mathbb Z}$ and refer to it as the (full) weight
lattice. We define the dual lattice $\mathfrak t_\mathbb Z$ by
\begin{equation}\label{labEqnDualLattice}
  \mathfrak t_\mathbb Z = \{ v \in \mathfrak t :
    u(v) \in \mathbb Z \text{ for all }
    u \in \mathfrak t^*_\mathbb Z \} .
\end{equation}
and obtain the pairing
\begin{equation}\label{labEqnLatticePairing}
  \langle \cdot, \cdot \rangle :
    \mathfrak t^*_\mathbb Z \times \mathfrak t_\mathbb Z
    \rightarrow \mathbb Z ,\quad
    \langle u, v \rangle = u(v) .
\end{equation}
Using the identification explained above, if the bundle $L$
is equipped with a lifting of the $T$-action, then $\Gamma (X,L)$ is
described as a $T$-representation space by the set of weights in
$\mathfrak {t}^*_{\mathbb Z}$ which lie in a certain polygonal
region which is defined by the geometry of $X$ as a $T$-space (see
$\S\ref{background}$).

\subsubsection* {Sequences of eigensections}

The purpose of our work is to explain certain asymptotic phenomena
for sequences $(s_N)$ of sections where $s_N\in \Gamma (X,L^N)$.
Having fixed the base point $s_0$ with $s_0(x_0)=:1$, we have the
base point $s_0^N$ for $\Gamma (X,L^N)$ with $s_0^N(x_0)=:1^N$. Thus
we have the correspondence between eigensections and linear
functions in $\mathfrak {t}^*_{\mathbb Z}$ at that level as well.
Recall this is given by using the lifting of the $T$-action via
$s_0^N$, noting that a given eigensection $s$ satisfies form
$s(tx_0)=\chi (t)s_0^N(tx_0)$ and expressing $\chi $ as $e^{2\pi
i\alpha }$. The corresondence is then defined by $s\mapsto \alpha $.

\bigskip\noindent
We wish to understand the asymptotic behavior of a sequence of
sections $(s_N)$ where the individual elements $s_N\in
\Gamma(X,L^N)$ are chosen so that the associated weights approximate
a ray $R(\xi):=\mathbb R^{\ge 0}.\xi $ defined by $\xi \in \mathfrak
{t}^*_{\mathbb Z}$. If $\alpha _N$ is the integral weight associated
to $s_N$, then one says that $(s_N)$ approximates $R(\xi)$ at
infinity if
\[
  \alpha _N=N.\xi + O(1)\,.
\]
It should be underlined that the entire discussion depends on the
choice of the lifting of the $T$-action to $L$. For example, if
the action is lifted via another base eigensection $\hat {s}_0$
which is associated to $s_0$ by the character $\chi _{\hat \alpha}$,
then a sequence $(s_N)$ approximates the ray $R(\xi)$ with respect
to the base section $s_0$ if and only if it approximates the
ray $R(\xi-\hat {\alpha})$ with respect to the base section 
$\hat {s}_0$.  

\bigskip\noindent
Our main result states that for every ray $R(\xi )$ and every
sequence $(s_N)$ which approximates $R(\xi)$ at infinity the
sequence 
$$
\vert \varphi _N\vert_h ^2:=\frac{\vert s_N\vert_h^2}{\ \ \Vert s_N\Vert_{L_2}}
$$
of probability densities
converges with precise estimates of both $\vert s_N\vert _h^2$ and
$\Vert s_N\Vert^2_{L_2}$ to the integration current of a 
certain $T_{\mathbb R}$-orbit $M$.  In fact $M$ is the set where
a certain canonically associated strictly plurisubharmonic function
$f:X\to \mathbb R^{\ge 0}\cupdot \{\infty \}$ takes on its minimum.
This function arises as follows.

\bigskip\noindent
Given $(s_N)$ one studies the strictly plurisubharmonic functions
$f_N:=-\frac{1}{N}\log \vert s_N\vert^2_h$.  It is a simple matter
to check that these converge (locally on compact subsets)
to a smooth strictly plurisubharmonic function $f$ on the open $T$-orbit.  
However, simple examples show that they do not necessarily 
converge on $X$, even outside the zero sets of the $s_N$.  
However, there exist tame
sequences $s_N'$ which approach the same ray at infinity so that
the polar sets of the associated functions $f_N'$ stabilize for
large $N$ as one ample divisor $Y$ with $f_N'\to f'$ uniformly on compact 
subsets on the complement $X\setminus Y$.  Since $f'=f$ on the
open orbit, we may define $f$ independent of the tame sequence
by continuation to $X$ simply by continuity. The role of $f$
is emphasized by the following result.
\begin{theorem}\label{labThmSPSF}
The smooth strictly plurisubharmonic function $f$ is an exhaustion
of $X\setminus Y$ which takes on its minimum exactly on the
$T_{\mathbb R}$-orbit $M$ which is a strong deformation retract 
of $X\setminus Y$.
\end{theorem}
The $T_{\mathbb R}$-orbit $M$ is contained in the closed $T$-orbit 
$\mathcal O_\tau$ in $X\setminus Y$. Both $Y$ and $\mathcal O_\tau$
are determined by $\xi $ by the combinatorial polyhedral geometry
associated to $X$ ($\S\ref{labSecTameSeq}$).  
The location of $M$ in $\mathcal O_\tau$
varies in a explicitly determined way as a function of $\xi $
and the (positive) metric $h$ ($\S\ref{labSecexistence}$). 

\bigskip\noindent
Normalizing $f$ such that $f|M = 0$ it follows that the sequence of
probabiliy densities ``localizes'' on $M$. This statement is made
precise in the following theorem

\begin{theorem}\label{labThmConvDirac}
Let $d\lambda$ be a probability measure on $X$. The sequence of
measures $|\varphi_N|_h^2 d\lambda$ converges to the Dirac measure
$\delta_M$ on $M$ in the weak sense. That is, we have
\[
  \int_X u \, |\varphi_N|_h^2 d\lambda \rightarrow 
    \int_M u \, dM
\]
for all continuous functions $u : X \rightarrow \mathbb R$.
\end{theorem}

For the proof of Theorem \ref{labThmSPSF} it is necessary to estimate
the pointwise asymptotic behavior of the probability densities. This
comes down to an analysis of the $L^2$-norms $\|s_N\|_{L^2}$. It
is only necessary to carry out the estimation locally near $M$.
This is possible because $f$ is in fact a Bott-Morse function
with its minimum on $M$. If $X$ is smooth, the relevant integrals
can be directly computed. In the singular case essentially the
same estimate holds, but the proof is more delicate (see the
appendix). The final result is of the form
\begin{equation}\label{labEqnPointwiseAsymp}
  |\varphi_N|_h^2 \sim N^\kappa e^{-Nf} .
\end{equation}
The speed of convergence, determined by the exponent $\kappa$,
depends on the position of $M$ in the stratification of $X$ given by
the $T$-action. Calculations are straightforward when $M$ lies in
the open and dense orbit $\mathcal O_0$. In this case $\kappa$
equals $\frac{1}{2}\dim X$. However, if $M$ lies in some boundary
component $\mathcal O_\tau$, the behavior is more subtle and more
technical effort is needed. Two problems arise: The singularities of
$X$ play a role and, as compared to the tame sequence, 
the sequence $(s_N)$ may show a irregular
behavior as it approaches the semiclassical limit given by $R(\xi)$.
In fact, as we will illustrate in an example later on, the pointwise
asymptotic behavior of $|\varphi_N|_h^2$ depends on $s_N$ and is not
uniquely determined by the ray $R(\xi)$.  

\bigskip\noindent
In $\S\ref{labSecStrictlyPSlimit}$ we derive a precise asymptotic formula for
tame sequences and discuss their relationship to arbitrary, possibly
non-tame sequences. It turns out that the tails of 
the distribution functions
defined by a tame sequence give upper estimates for the
orginal ones. This is the content of the following theorem:

\begin{theorem}\label{labThmDistFunc}
Let $(s_N)$ be a sequence of sections approximating a ray $R(\xi)$
at infinity and let 
$$
D_N(t):=\mathrm {Vol}\{x\in X; \vert \varphi_N\vert_h^2>t\}
$$
be the tail of the associated distribution function.  If 
$(s_N')$ is an associated tame sequence, then
\[
  D_N(t) \leq D_N'(t) \sim \Bigl( \frac{\log N}{N} \Bigr)^\kappa ,
\]
where $\kappa = \operatorname{codim} \mathcal O_\tau + \frac{1}{2}
\mathcal \dim \mathcal O_\tau$.
\end{theorem}
For tame sequences, the precise nature of this asymptotic statement
and the methods of \cite{STZ} facilitate scaling the probability
distribution to obtain a universal probability distribution. This
will appear in \cite{S}. It should also be mentioned that for 
any given sequence the difference between $D_N(t)$ and $D'_N(t)$ is only of
finite order in $N$. 

\subsubsection*{Previous results}

In \cite{STZ} the authors derive formul\ae{} for the pointwise
asymptotic behavior of the probability densities $|\varphi_N|_h^2$
and the distribution functions $D_N(t)$ in the case of $X$ being
smooth and embedded in projective space with the bundle metric
$h$ being the restriction of the Fubini-Study metric. If the
localization manifold $M$ is located on the boundary of the open
orbit they only consider asymptotic sequences of a special type: The
element $\xi$ defining the ray $R(\xi)$ is assumed to be integral
and the asymptotic sequence of characters $\alpha_N$ is assumed to
be $\alpha_N = N\xi$. This is a special case of a tame sequence as
it is considered in the present paper.

\bigskip\noindent
In the more recent paper \cite{SZ} the authors also deal with the
smooth case. In particular, they derive asymptotic developments
of the value of $\vert \varphi_N\vert^2_N$ at the momentum map
preimage of $\frac{\alpha_N}{N}$.  This is carried out without 
reference to a particular choice of a sequence $(s_N)$ but instead
depends on the location of $\frac{\alpha_N}{N}$ in the momentum
map image $P$.  This is a delicate matter when $\frac{\alpha_N}{N}$
approaches a face of the boundary of $P$ (see $\S6.3$). In our relatively
simple considerations, a similar phenomenon arises. In our corresponding
situation where $M\subset \mathcal O_\tau$ and $\tau\not=0$, the
sequence $(s_N)$ may or may not be tame. If it is not tame one 
can not expect a universal scaled probability density.  It should be
remarked that in the smooth case the limiting function $f$ 
can be explicitly computed in the relevant local coordinants.
This done in the course of the work in \cite{SZ} (see also \cite{S}).

\bigskip\noindent
In \cite{BGW} the smooth case is also considered.  There the authors
take advantage of the Delzant construction which realizes $X$ as
a certain GIT-quotient of the set of stable points of a linear
torus action on $\mathbb C^d$.  The flat metric on the trivial bundle
on $\mathbb C^d$ defines a positive Hermitian metric on the push-down
of the trivial bundle. In the toric setting they derive explicit
formulas for the \emph{stability function} which compares the two
metrics. In a general situation (see $\S5$) they make a quantitative
comparison of the $L_2$-norms.  This allows them to apply previously
obtained results (Lemma 8.1) to derive a formula for 
the universal scaled probability distribution under special conditions
on $\alpha_N$ (see $\S8$).
\section {Preparation on toric varieties}\label {background}
Here we begin by presenting certain background information for dealing with
the combinatorial side of the theory of toric varieties.  Our main
observation is that associated to any ray there is 
a tame sequence. 

\subsubsection*{One-parameter subgroups}

Recall that by definition there is a open and dense $T$-orbit
$\mathcal O_0$ in $X$. In order to understand the structure of the 
boundary $\mathrm {bd}(\mathcal O_0)$ one considers 
algebraic 1-parameter subgroups $\lambda :\mathbb C^* \rightarrow T$.
In a systematic way one determines each \emph{boundary orbit} as
the orbit of a limit point
\begin{equation}\label{labEqnOneParameterGroup}
  x_0^\tau := \lim_{z \to 0} \lambda(z) x_0 \in X 
\end{equation}
for $\lambda $ appropriately chosen.  Many considerations
are in this way reduced to the 1-dimensional case. Since $T =
(\mathbb C^*)^m$, 1-parameter subgroups are in 1-1 correspondence
with integral vectors $v = (n_1, \ldots, n_m) \in \mathbb Z^m$
by setting
\[
  \lambda_v : \mathbb C^* \rightarrow T, \quad
    \lambda_v(z) = (z^{n_1}, \ldots, z^{n_m}) . 
\]
We therefore may regard the integral lattice $\mathfrak t_\mathbb Z$ 
defined in \eqref{labEqnDualLattice} as the space of 1-parameter 
subgroups of $T$. Using the pairing $\langle \cdot, \cdot
\rangle : \mathfrak t^*_\mathbb Z \times \mathfrak t_\mathbb Z
\rightarrow \mathbb Z$ from \eqref{labEqnLatticePairing} we
characterize one-parameter groups by
\[
  \chi_\alpha(\lambda_v(z)) = z^{\left< \alpha, v \right>}
\]
for a character $\chi_\alpha = e^{2\pi i \alpha}$ with $\alpha \in
\mathfrak t^*_\mathbb Z$.

\subsection{Orbit structure}\label{labSecTorusCombinatorics}

In the theory of toric varieties (see \cite{F}) one associates in a
one-to-one fashion to each $T$-orbit $\mathcal O \subset X$ a set
$\tau \subset \mathfrak t_\mathbb R$ called a \textit{strongly
convex rational polyhedral cone}. Such a cone is by definition the
set of all convex combinations of a set of integral vectors $v_1,
\ldots, v_r \in \mathfrak t_\mathbb Z$, called the
\textit{generators of $\tau$}, i.e.
\[
  \tau = \{ \lambda_1 v_1 + \cdots + \lambda_r v_r : \lambda_j \geq
    0\} =: \langle v_1, \dots, v_r \rangle_{\mathbb R^{\geq 0}} 
\]
such that $\tau \cap (-\tau) = \{0\}$. The vectors $v_j$ are chosen
in way such that the one-parameter groups defined by
$\mathrm{Int}(\tau) \cap \mathfrak t_\mathbb Z$ close up in the
orbit $\mathcal O_\tau$; that is, if $v \in \mathfrak t_\mathbb Z$
is in the relative interior of $\tau$, then $\lim_{z \to 0}
\lambda_v(z)x_0 = x_0^\tau$ is an element of $\mathcal O_\tau$. The
collection of all such cones $\tau$ constitute a \textit{fan
$\Sigma(X)$}. Fans and their exact relation to the orbit structure
of $X$ are studied in detail, e.g., in \cite{F}. What is important
for us in the following are two basic facts:

\begin{enumerate}

\item Each $T$-invariant complex hypersurface $Y_j$ corresponds to a
one-dimensional cone $\tau_j = \left<v_j\right>_{\mathbb R^{\geq
0}}$ with $v_j \in \mathfrak t_\mathbb Z$.

\item A cone $\sigma$ of maximal dimension corresponds to a
$T$-fixed point $x_\sigma$.

\end{enumerate}

\subsubsection*{Parameterization of eigensections}

Let $Y_1, \ldots, Y_\ell$ be the $T$-invariant hypersurfaces in $X$.
The base section $s_0$ chosen in \S{}\ref{labSecTorusLifting}
defines a $T$-invariant divisor
\begin{equation}\label{labEqnBaseDivisor}
  D = \sum_j a_j Y_j . 
\end{equation}
Of course $L=L(D)$. Each hypersurface $Y_j$ is given by a one-dimensional cone $\tau_j =
\left< v_j \right>_{\mathbb R^{\geq 0}}$ in the fan of $X$. We
define the set
\begin{equation}\label{labEqnPolytope}
  P_D = \cap_{j=1}^\ell \{u \in \mathfrak t^* :
    \langle u,v_j\rangle \ge -a_j\} .
\end{equation}
The following Proposition characterizes holomorphic eigensections in
terms of this polyhedron; it is standard in the theory of toric
varieties.

\begin{proposition}\label{labThmPolytopeSections}
The spaces of sections
$\Gamma(X, L^N)$ are given by
\[
  \Gamma(X, L^N) = 
    \oplus_\alpha \mathbb C s_\alpha ,
    \text{ where }\alpha \in NP_D \cap \mathfrak t^*_\mathbb Z 
\]
and the $s_\alpha $ are $T$-eigensections with 
$t(s_\alpha)=\chi_\alpha(t)s_\alpha$. Furthermore, if 
$D$ is ample, then $P_D$ is a strictly convex polytope.
It is bounded precisely when $X$ is compact.
\end{proposition}

\begin{remark}
Since in our case $L$ is very ample, one can explicitly construct
the Kodaira embedding of $X$ by using the polytope $P_D$. This is
intially given by 
$x\mapsto [s_\alpha(x)]$ where the $s_\alpha $ are the
distinguished $T$-eigensections in $\Gamma (X,L)$. Recalling that
$s_\alpha(tx_0)=\chi_\alpha (t)s_0(tx_0)$ on the open orbit,
the embedding can be written as $t(x_0)\mapsto [\chi_\alpha (t)]$.  
In other words $t(x_0)$ is mapped to $t[1:\ldots : 1]$ and 
$X$ is identified with the closure of the $T$-orbit $T.[1:\ldots :1]$.  
In \cite{STZ} toric varieties are actually defined this way. 
To ensure smoothness when taking the closure they require the polytope 
$P_D$ to be Delzant. \qed
\end{remark}

\subsection{Existence of asymptotic sequences}

As an application of Proposition \ref{labThmPolytopeSections} we
show that, given a ray $R(\xi)$ there exists a sequence of
holomorphic eigensections $s_N$ approximating $R(\xi)$ at infinity
if and only if $\xi$ is an element of the polytope $P_D$. The
neccessity can be proved as follows.
\begin{proposition}\label{labThmXiInPD}
If $(s_N)$ approximates $\mathbb R^{\geq 0}\xi$ at infinity, then 
$\xi \in P_D$.
\end{proposition}

\begin{proof}
From the definition of a sequence of characters approximating a ray
$R(\xi)$ we have $\alpha_N/N = \xi + O(N^{-1})$ . Since each $s_N$
is holomorphic, the corresponding character $\alpha_N$ must lie in
the polytope $NP_D$, whose defining equations are given by
\eqref{labEqnPolytope}. For the vector $\xi$ we now have
\[
  \langle \xi, v_j \rangle =
    \langle \frac{\alpha_N}{N}, v_j \rangle - 
      \langle \frac{\alpha_N}{N} - \xi, v_j\rangle
    \geq -a_j + O(N^{-1}) .
\]
The last inequality is true for all $N$, thus $\langle \xi, v_j
\rangle \geq -a_j$ for all $j$, and hence $\xi \in P_D$.
\end{proof}

The converse is also true but requires more effort.

\begin{proposition}\label{labThmAsympExistence}
If $\xi \in P_D$, then there exists a sequence $\alpha_N \in NP_D
\cap \mathfrak t^*_\mathbb Z$ such that $\alpha_N = N\xi + O(1)$ for
all $N \in \mathbb N$.
\end{proposition}
Since for every $\xi \in \mathbb R^m$ there exists a sequence
$(\alpha_N)$ of integral points with $\Vert N\xi-\alpha_N\Vert <1$, 
the proposition is a consequence of the following projection argument.
\begin{lemma}\label{labThmProjectionLemma}
Let $F$ be a face of the polytope $P_D$ and
$\xi$ be in the relative interior of $F$. If $\alpha_N \in \mathfrak
t^*_\mathbb Z$ is any integral sequence such that $\alpha_N = N\xi +
O(1)$, then there exists a sequence $\alpha_N'$ having that same
property but in addtion $\alpha_N'/N \in F \cap \mathfrak t^*_\mathbb
Z$ for almost all $N$.
\end{lemma}
\begin{proof}
Suppose $\operatorname{codim} F = 0$. Then $\xi$ lies in the
interior of $P_D$ which is open. Since the sequence $\alpha_N/N$
converges to $\xi$, the point $\alpha_N/N$ will lie in the interior
of $P_D$ for big $N$ and we can simply set $\alpha_N' = \alpha_N$.
The other extreme case is $\operatorname{codim} F = \dim P_D$, which
means that $F$ is a vertex of $P_D$. In this case $\xi$ is integral
and we can set $\alpha_N' = N\xi$. The remaining case is
characterized by $0 < \operatorname{codim} F < \dim P_D$. Here 
we replace $\alpha_N/N$ by its projection onto the the face
$F$. For this choose integral vectors $w_1, \ldots, w_k$ and a
vertex $\alpha_\sigma$ of $F$ such that $F \subset \alpha_\sigma +
\mathrm{span}_{\mathbb R}\{w_1, \ldots, w_k\}$ and define the
projection
\[
  \alpha_N' = N(\alpha_\sigma + \sum_{j=1}^k \langle 
    \frac{\alpha_N}{N} - \alpha_\sigma, w_j \rangle w_j).
\]
By writing $\xi = \alpha_\sigma + \sum_j \langle \xi -
\alpha_\sigma, w_j\rangle w_j$ one immediately sees
\[
  \|\frac{\alpha_N'}{N} - \xi\| \leq \frac{C}{N} 
    \sum_{j=1}^k \|w_j\| = O(\frac{1}{N}) .
\] 
Hence, the sequences $\alpha_N$ and $\alpha_N'$ approximate the same
ray $\mathbb R^{\geq 0}\xi$. Furthermore, since $\xi$ lies in the
relative interior of $F$, the elements $\alpha_N'/N$ will also be in
$F$ for almost all $N$.
\end{proof}
\section {Strictly plurisubharmonic limit
functions}\label{labSecStrictlyPSlimit}

The goal of this section is to prove the existence of a certain strictly
plurisubharmonic limit function $f : X \rightarrow \mathbb R \cupdot
\{\infty\}$ which is canonically associated to a semiclassical ray 
$R(\xi)$. This function provides the main tool for studying the 
asymptotic behavior of the probability densities $|\varphi_N|_h^2$ and 
the tails of the distribution functions $D_N(t)$.

\bigskip\noindent
On the open orbit we may simply define $f = \lim_{N \to \infty}
-\frac{1}{N} \log|s_N|_h^2$. However, in order to extend $f$ to 
$X$ we need to know more about the behavior of the $s_N$ at
the boundary. The following example shows that this might be
quite irregular.

\begin{example}\label{labExample}
Let $X=\mathbb P_1$ and $L=H$ be the hyperplane section bundle
equipped with its standard metric $h$. If $[z_0:z_1]$ are standard
homogeneous coordinates and sections $s\in \Gamma (X,L)$ are
regarded as linear functions $\ell (z_0,z_1)$, then
\[
  \vert s\vert^2_h=\frac{\vert \ell (z_0,z_1)\vert^2}
    {\vert z_0\vert^2+\vert z_1\vert ^2}\,.
\]
More generally if we equip $L^N$ which the associated tensor power
metric and a section $s\in \Gamma (X,L^N)$ is represented by a
homogeneous polynomial $P$ of degree $N$, then
\[
  \vert s\vert^2=\frac{\vert P(z_0,z_1)\vert^2}
    {(\vert z_0\vert^2+\vert z_1\vert ^2)^N}\,.
\]
Let us begin with the sequence defined by $s_N=z_0^N$. It
corresponds to the sequence of characters $\alpha_N = 0$ for all
$N$. For the probability density we must compute the 
integral
\[
  \int_{\mathbb P_1}\vert s_N\vert^2=
    \int_{\mathbb C}\frac{1}{(1+\vert z\vert^2)^N}=
    \int_0^\infty \frac{1}{(1+r^2)^N}rdr 
    \sim \frac{1}{N-1}\,.
\]
If we replace $s_N$ by the sequence defined by the homogeneous
polynomials $z_0^{N-1}z_1$, belonging to the sequence of characters
$\alpha_N = 1$ for all $N$, then
\[
  \int_{\mathbb P_1}\vert s_N\vert^2=
    \int_{\mathbb C}\frac{\vert z\vert^2}{(1+\vert z\vert^2)^N}
    \sim \frac{1}{(N-1)(N-2)}\,.
\]
Thus the integrals are asymptotically different. The
situation is even worse if we allow the sections to jump around like
in the following example:
\[
  s_N = \begin{cases}
    z_0^N & \text{for } N \text{ odd}, \\
    z_1 z_0^{N-1} & \text{otherwise} .
  \end{cases}
\]
Nevertheless, in all of the above cases the probability density
$\vert s_N\vert^2/\Vert s_N\Vert^2_{L_2}$ converges in measure to
the Dirac measure of the point $[1:0]$.\qed
\end{example}

\subsection{Tame sequences}\label{labSecTameSeq}

We can avoid the kind of problems illustrated above by replacing the
sequence $(s_N)$ by a new sequence $(s_N')$ that approximates
the same ray, but whose vanishing orders at the boundary
can be better controlled. For the construction of $(s_N)$ we
first note that the asymptotic vanishing order $\mathrm{ord}_{Y_j}(s_N)$
along a boundary hypersurface $Y_j$ is well-defined and completely
determined by $R(\xi)$.
\begin{lemma}\label{labThmExistenceOfkj}
There exist non-negative real numbers $k_1, \dots, k_\ell$ such that
for each boundary component $Y_j$ we have
\begin{equation}\label{labEqnExistenceOfkj}
 \lim_{N \to \infty} \frac{1}{N} \mathrm{ord}_{Y_j}(s_N) = k_j .
\end{equation} 
\end{lemma}
\begin{proof}
A $T$-invariant hypersurface $Y_j$ is determined by a
1-dimensional cone in the fan of $X$. Such a cone is generated by
a vector $v_j \in \mathfrak t_\mathbb Z$. Thus, $Y_j$ determines a
1-codimensional face $H_j$ of the polytope $P_D$ in the following
way (see also \eqref{labEqnPolytope})
\begin{equation}\label{labEqnFacet}
  H_j = P_D \cap \{u \in \mathfrak t^* : 
    \langle u, v_j \rangle = -a_j\} .
\end{equation}
Choosing a vertex $\alpha_\sigma \in H_j$ it follows that the
corresponding eigensection $s_\sigma$ does not identically vanish on
$Y_j$. Here we write $s_N =
s_{N, \sigma} s_\sigma^N$ where $s_{N,\sigma}$ is a $T$-equivariant
holomorphic function transforming by the character $\chi_N
\chi_\sigma^{-N}$. Its order along $Y_j$ is given by $\langle
\alpha_N - N\alpha_\sigma, v_j \rangle$. Since the sequence
$\alpha_N/N$ converges to $\xi$ we have
\begin{equation}\label{labEqnAsympOrder}
  \begin{split}
    \frac{1}{N} \mathrm{ord}_{Y_j}(s_N) = 
      \langle \frac{\alpha_N}{N} - \alpha_\sigma, v_j \rangle 
      \rightarrow \langle \xi - \alpha_\sigma, v_j \rangle \\ 
      = \left< \xi, v_j \right> + a_j
      =: k_j .
  \end{split}
\end{equation}
This proves the existence of the numbers $k_j$ for any sequence
$s_N$ approximating the ray $R(\xi)$ at infinity.
\end{proof}

\begin{proposition}\label{labThmTame}
There exists a sequence $(s_N')$ such that if $k_j = 0$, then
$\mathrm{ord}_{Y_j}(s_N') = 0$ for almost all $N \in \mathbb N$.
\end{proposition}

\begin{proof}
From equation \eqref{labEqnAsympOrder} we have $k_j = \langle \xi,
v_j \rangle + a_j$, where the $a_j$ are the coefficients defining
$P_D$, see \eqref{labEqnPolytope}. Thus, $k_j = 0$ if and only if
$\xi$ lies in the hyperplane $H_j$ defined by equation
\eqref{labEqnFacet}. A collection $\{H_j\}$ of such hyperplanes
determines a face of the polytope
\begin{equation}\label{labEqnFace}
  \xi \in F = P_D \cap H_1 \cap \ldots \cap H_r .
\end{equation}
Hence, the condition that $\mathrm{ord}_{Y_j}(s_N) = 0$ whenever
$k_j = 0$ is equivalent to saying that $\alpha_N$ is an element of
the same face as $\xi$. By Lemma \ref{labThmProjectionLemma} there
exists a sequence $\alpha_N'$ approximating the same ray $R(\xi)$
but having in addition the property $\alpha_N' \in F \cap \mathfrak
t_\mathbb Z^*$ for allmost all $N$. This proves the assertion.
\end{proof}

\begin{definition}
A sequence $(s_N)$ which approximates a ray at infinity
said to be \emph{tame} if it fulfills the properties of Proposition
\ref{labThmTame}. The union $Y$ of the
hypersurfaces $Y_j$ with $k_j>0$ is called the \emph{limiting support} of
the sequence.
\end{definition}

\begin{remarks}
1. If $\xi$ lies in the interior of the polytope $P_D$,
then every asymptotic sequence approximating the ray $R(\xi)$ is
tame. 2. If $\alpha'_N$ is the tame sequence constructed from a given
one $\alpha_N$, then the difference $\alpha'_N - \alpha_N$ is
uniformly bounded in $N$.
\end{remarks}

By inspecting the proofs of Lemma \ref{labThmExistenceOfkj} and
Proposition \ref{labThmTame} we can give a description of the
limiting support $Y$ of a tame sequence in terms of the polytope
$P_D$.

\begin{proposition}
If $\{H_j\}$ is the set of all 1-codimensional faces of $P_D$ which
do not contain $\xi$ and $\{Y_j\}$ is the set of associated
$T$-invariant hypersurfaces, then the limiting support of a tame sequence
is given by $Y = \cup_j Y_j$.
\end{proposition}


\subsubsection* {On the complex geometry of $\mathrm {X\setminus Y}$}
If $(s_N)$ is a tame sequence with asymptotic support $Y$,
then, since $L$ is ample, the $T$-invariant open set $X\setminus Y$
is an affine variety. Since $T$ has an open orbit in $X\setminus Y$,
it follows it possesses only the constant invariant holomorphic functions
and as a result it possesses a unique closed orbit $\mathcal O_\tau$.

\bigskip\noindent
Let us explain how $\xi $ determines $\mathcal O_\tau $ in the
combinatorial language of toric varieties. For this recall that the face
$F$ is defined by the condition that $\xi $ is in its relative interior.
In the following way $F$ is associated to a cone $\tau $ in the fan
$\Sigma (X)$. By definition, $F$ is the 
intersection of $P_D$ with a set of supporting hyperplanes. That is, if $I
\subset \{1, \ldots, \ell\}$ is an index set, then $F$ is defined by
\begin{equation}\label{labEqnFaceDef}
  F = P_D \cap \cap_{i \in I} H_i 
    \quad\text{where}\quad
    H_i = \{ u \in \mathfrak {t}^*_{\mathbb R}: \langle u, v_i \rangle = -a_i \} .
\end{equation}
Since $D$ is very ample, the vectors $v_k$ with $k \in I$ define a
cone $\tau$ in the fan of $X$.  This cone, regarded as a fan, defines
the affine toric variety $X\setminus Y$. In particular, the relative
interior of $\tau $ corresponds to the closed (dimension-theoretically
minimal) orbit in $X\setminus Y$. \qed

\bigskip\noindent
Again in the setting of a tame sequence we observe that for 
$N$ sufficiently large the functions $f_N$ are $T_{\mathbb R}$-invariant 
strictly plurisubharmonic exhaustions of $X\setminus Y$. In the next 
section it is shown that they converge to a function $f$ which is
likewise a smooth strictly plurisubharmonic
exhaustion.  Thus the following is relevant for our considerations.
\begin {theorem}  
Let $Z$ be a Stein space equipped with a holomorphic
action of a reductive group $G$ which is the complexification of
a maximal compact subgroup $K$ and let $\rho :X\to \mathbb R^{\ge 0}$
be a smooth proper $K$-invariant strictly plurisubharmonic exhaustion.
Then its minimum set $M:=\{\zeta\in Z; \rho (\zeta)=
\mathrm{min}\{\rho(z);z\in Z\}\}$ consists of a single $K$-orbit which
is contained in the closed $G$-orbit in $Z$.
\end {theorem}
This result is a special case of Corollary 1 in $\S5.4$ of \cite{H}.
It is one of the basic first steps for the construction of
the analytic Hilbert quotient by the method of K\"ahlerian reduction
(see \cite {HH1}).  It should also be mentioned that using the
gradient of the norm of the associated moment map one shows
that $M$ is a strong deformation retract of $Z$ (\cite {HH2}).
Although we apply these results in the case of an affine variety,
the plurisubharmonic functions at hand are only smooth and are
not of the type where the algebraic theory can be applied
(See \cite{N} for basic results in the situation where $Z$ is
equivariantly embedded in a representation space and $\rho $
is the restriction of a $K$-invariant norm-function.).

\subsubsection* {Product structure of $\mathbf{X\setminus Y}$}
Continuing in our special setting where $\mathcal O_\tau $ is as
above, we let $T_\tau $ be the connected component at the
identity of the isotropy group $T_{x_\tau}$ of any base point in 
$\mathcal O_\tau$.  This can be alternatively described as the 
connected component at the identity of the ineffectivity of the 
$T$-action on $\mathcal O_\tau$.
Let $T'$ be a complementary complex torus in $T$, i.e., $T=T_\tau \times T'$.

\bigskip\noindent
Let us choose $x_\tau$ so that there is a 1-parameter subgroup
$\lambda $ with $\lim_{t\to 0}\lambda (t)(x_0)=x_\tau$. Clearly
$\lambda (\mathbb C^*)\in T_\tau$ and thus $x_\tau$ is in the
closure $\mathrm {cl}(T_\tau.x_0)=:\Sigma_\tau$.  In fact $\{x_\tau\}$ is the
unique closed orbit in this closure and therefore the restriction
to $\Sigma_\tau$ of \emph{any} $T_{\mathbb R}$-invariant strictly 
plurisubharmonic exhaustion has exactly $\{x_\tau\}$ as its minimizing
set.
\begin {proposition}
The map $\alpha: T'\times \Sigma_\tau\to X\setminus Y$,
$(t',x)\mapsto t'(x)$, establishes an $T$-equivariant isomorphism
$\mathcal O_\tau\times \Sigma_\tau\to X\setminus Y$.
\end {proposition}
\begin {proof}
To be precise we replace $T'$ by its quotient by its finite
ineffecitivity on $\mathcal O_\tau $.  Thus $T'$ acts freely
on $\mathcal O_\tau$. If $t'_1(x)=t_2'(x)$ for some $t_1',t_2'\in T'$
and $x\in \Sigma_\tau$, then $t':=t_1'(t_2')^{-1}$ is such that
$\Sigma_\tau \cap t'(\Sigma_\tau)\not=\emptyset $.  But 
this (closed) interesection is $T_\tau$-invariant and thus it contains the
closed $T_\tau$-orbit $\{x_\tau\}$.  It therefore follows that
$t_1'=t_2'$. Consequently $\alpha $ is an isomorphism onto an
open neighborhood of $\mathcal O_\tau$ which is equivariant with
respect to $T=T_\tau \times T'$.  Since every $T$-orbit in
$X\setminus Y$ has a point of $\mathcal O_\tau $ in its closure, it
follows that $\alpha $ is surjective.
\end {proof}
It should be mentioned that this product decomposition can
be proved by purely combinatorial means (see \cite{F}).


\subsection{Existence of the limit function}\label{labSecexistence}

Recall that our goal is to understand the limiting properties
of the probability density funtion $\vert \varphi_N\vert_h^2$.
The following is the first main step in this direction.
\begin {proposition}\label {tame sequence}
If $(s_N)$ is a tame sequence which approximates a ray $R(\xi)$ 
at infinity with limiting support $Y$, then the associated
sequence
\[
  f_N=-\frac{1}{N}\log \vert s_N\vert^2_h
\]
converges uniformly on compact subsets of $X\setminus Y$ to a smooth
strictly plurisubharmonic function $f$.
\end {proposition}

{\bf Zusatz.} {\it Any sequence approximaiing the same ray at
infinity converges uniformly on compact subsets of the open orbit to
$f$. Thus $f$ could be altermatively defined as the extension by
continuity to $X\setminus Y$ of this limit function.}

\begin{remarks}
First, observe that by using the local description of $s_N$ near its
supporting hypersurfaces one shows $f$ can be extended by continuity
to a function $f:X\to \mathbb R \cupdot \{\infty\}$ by defining it
to be identically $\infty $ on $Y$. This extended function is
strictly plurisubharmonic as an $\mathbb R \cupdot \{\infty\}$
valued function. Secondly, as the reader will note in the proof, the
convergence $f_N\to f$ is locally given by the convergence of the
sequence $(s_N^{1/N})$ of holomorphic functions. Thus the
convergence in the $C^\infty$-topology is also guaranteed. Finally,
if $\widehat {s}_N$ is any sequence approximating $R(\xi)$, then on
the open orbit
\[
  f_N - \widehat{f}_N = 2\mathrm{Re}
    \Bigl( \frac{\beta_N}{N} \Bigr) 
\]
where $\beta _N$ is a sequence of linear functions which are
contained in a bounded set. Consequently, on the open orbit $f$ and
$\widehat {f}$ agree. As a result the limiting function defined by a
tame sequence is unique: Take any sequence which approximates $R(\xi
)$ at infinity and define $f$ to be the function on $X$ which is
obtained by extending the uniquely defined function on the open
orbit to all of $X$. Below we will also show that that the limiting
measure of the probability densities $\vert \varphi_N\vert_h^2$ also
only depends on the ray and not on the particular sequence which
approximates it at infinity. \qed
\end{remarks}

The following fact plays an essential role in our proof of
Proposition \ref{tame sequence}.

\begin {lemma}\label {saturation}
Let $Z$ be a compact toric variety of a group $T\cong(\mathbb
C^*)^k$ and $z_0\in Z$ be a $T$-fixed point. If $L$ is a very ample
$T$-line bundle on $Z$, it follows that up to constant multiples
there is exactly one eigensection which does not vanish at $z_0$.
\end {lemma}

\begin {proof}
For every point $z$ of the (Zariski open) saturation $\mathcal
S(z_0):=\{z\in T; z_0\in \mathrm {cl}(T.z)\}$ there is a 1-parameter
group $\lambda :\mathbb C^*\to T$ with $z_0$ in the (compact)
closure $C=\mathrm {cl}(\lambda (\mathbb C^*).z_0)$ in $Z$. Let
$\widehat {C}$ be the normalization of such a curve. Note that
$\widehat {C}\cong \mathbb P_1$ and that the lifted action of
$\mathbb C^*$ has two fixed points, one over $z_0$ and another
$\widehat {z}_1$ over some other point in $C$. The pull-back
$\widehat L\to \widehat C$ of the restriction of $L$ to $C$ is
isomorphic to some positive power of $H^k$ of the hyperplane section
bundle. If $s_1,s_2\in \Gamma (Z,L)\setminus \{0\}$ are
eigensections, neither of which vanishes at $z_0$, then the lifts
$\widehat {s}_1$ and $\widehat {s}_2$ of their restrictions to $C$
are $\mathbb C^*$-eigensections of $\widehat {L}$ on $\widehat {C}$
which only vanish at $\widehat {z}_1$. Thus $\widehat {s}_1$ is a
constant multiple of $\widehat {s}_2$ and consequently $s_1\vert C$
is a constant multiple $s_2\vert C$. Since this holds for every
curve $C$ constructed in this way and the constant is uniquely
determined by $s_1(z_0)$ and $s_2(z_0)$, the desired result follows.
\end {proof}
We apply this Lemma to restrictions of eigensections to the
closures of fibers of the $T_\tau$-invariant projection
map $q:\Sigma_\tau\times \mathcal O_\tau\to \mathcal O_\tau$.
\begin {corollary}
If the base section $s_0$ is chosen so that $s_0(x_\tau)\not=0$
and $s_N$ is a $T$-eigensection in $\Gamma (X,L^N)$ which also 
does not vanish at $x_\tau $, then the restrictions of 
$s$ and $s_0$ to any $q$-fiber agree. In particular, the
group $T_\tau$ is in the kernel of the character $\chi $
associated to $s$ by the base eigensection $s_0$. In particular,
the relation $s(tx_\tau)=\chi (t)s_0(tx_\tau)$ holds on $\mathcal O_\tau$
\end {corollary}
\begin {proof}
This follows immediately from the above proposition by applying
it in the case of the $T_\tau$-action on the closure of a $q$-fibers.
It is applicable because the closed $T_\tau$-orbit is its fixed point
which is the intersection point of the $q$-fiber with $\mathcal O_\tau$.
Furthermore, $s_N(x_0)=s_0(x_0)^N$; so 
$s_N\vert \Sigma _\tau= cs_0\vert \Sigma _\tau $ with $c=1$.  Equality
on the other fibers follows from the fact that $T$ acts transitively
on the set of these fibers.
\end {proof}

Of course the original base section $s_0$ may vanish at $x_\tau$.
If it does, using the fact that $L$ is very ample, we know
that there \emph{is} a $T$-eigensection $\hat {s}_0$ which does
not vanish there.  As we have already noted, changing from
$s_0$ to $\hat {s}_0$ only has the effect of translating
$R(\xi)$ to $R(\xi-\hat {\alpha})$ where the character associated
to $\hat {s}_0$ for the base point $s_0$ is $\chi _{\hat \alpha}$.

\bigskip\noindent
{\it Proof of Proposition \ref{tame sequence}}: It follows 
from the above corollary and the remark that a base change
has no influence on the discussion that we must 
only prove this for the sequence $(s_N\vert \mathcal O_\tau)$.  
But this is just the (possibly lower-dimensional) case of the open orbit!  
In that case
$$
f_N=-\frac{1}{N}\log \vert \chi_{\alpha_N}\vert ^2+\log \vert s_0\vert_h^2
$$
and the desired convergence is guaranteed by the fact that
$\frac{\alpha_N}{N}=\xi +O(N^{-1})$.

\bigskip\noindent
Note that since
$-\frac{1}{N}\log \vert \chi_{\alpha_N}\vert^2$ converges, given a point
in $\mathcal O_\tau$ we may choose 
a subsequence so that $(s_{N_k})^{\frac{1}{N_k}}$ converges locally
uniformly.  Since the local limit section is holomorphic, it
follows that the limit function $f$ is smooth and strictly plurisubharmonic.
\qed

\subsubsection* {Exhaustion property}
As we have seen above the functions $f_N$ and the limit $f$
are smooth strictly plurisubharmonic exhaustions of $X\setminus Y$.
Let $M:=\{z\in X\setminus Y; f(z)=\mathrm {min}\{f(x);x\in X\setminus Y\}\}$
and define $M_N$ to be the corresponding set for $f_N$.  We know
that $M_N=T_{\mathbb R}.x_N$ and $M=T_{\mathbb R}.x_0^\tau$ for points
$x_N,x_0^\tau\in \mathcal O_\tau$. Here we fix $x_0^\tau$ for the
discussion and normalize $f$ so that $f(x_0^\tau)=0$.  Let us
underline that the functions $f_N\vert \mathcal O_\tau$ and 
$f\vert \mathcal O_\tau$ have particularly strong convexity properties.
\begin {proposition}\label{convex exhaustion}
Let $\rho $ is smooth $(S^1)^k$-invariant strictly plurisubharmonic
function on $(\mathbb C^*)^k$ which takes on a local minimum
at a point $z_0$, then this is a global minimum, the
set
$M:=\{z\in X\setminus Y; \rho(z)=\mathrm {min}\{\rho(x);x\in X\setminus Y\}\}$
is the $(S^1)^k$-orbit of $z_0$ and $\rho $ is an exhaustion 
of $(\mathbb C^*)^k$.
\end {proposition} 
\begin {proof}
Using polar coordinates we write $\rho =e^{r(t)}$ where $r$ is a 
strictly convex function on $\mathbb R^k$ which takes on a
local minimum at $\log z_0$. 
\end {proof} 
\begin {corollary}
The points $x_N$ can be chosen in $M_N$ so that that they converge
to $x_0^\tau$.
\end {corollary}
\begin {proof}
If $U=U(M)$ is an arbitrary $T_{\mathbb R}$-invariant relatively
compact neighborhood of $M$ in $X\setminus Y$, then the restriction
of $f$ to the boundary of $U(M)\cap T.x_1$ is strictly larger than
$0$. Since $f_N$ converges uniformaly to $f$ on $U(M)$, this implies
that for $N$ sufficiently large $f_N$ attains a local minimum in the
interior of $U(M)$. But the strong convexity of $f_N$ (see the
argument above) implies that this local minimum is global, i.e.,
$M_N\subset U(M)$.
\end {proof}

\subsubsection*{Dependency on the metric and the ray}

The minimizing set $M$ of $f$ is contained in
the closed $T$-orbit $\mathcal O_\tau$ in $X\setminus Y$ where 
$Y$ and $\mathcal O_\tau$ are completely determined by the face
of $P_D$ which contains $\xi $.  The exact location of $M$ in
$\mathcal O_\tau$ depends on the metric $h$ and $\xi$.  With
respect to metrics we underline that we only consider those
which are positive in the sense that $-\log \vert s\vert^2_h$ is
strictly plurisubharmonic for any (local) holomorphic section $s$. 
If $\hat h=e^{\phi}h$ is such a metric and $\hat f$
is the associated strictly plurisubharmonic limit function
associated to a ray, then a direct calculation shows that
$\hat f=-\phi +f$.  This has the corresponding effect on the
minimizing set in $\mathcal O_\tau$ 

\bigskip\noindent
Concerning the dependence on $\xi $, recall that the convergence
$f_N\to f$ might require a change of the base eigensection
so that $\alpha _N=N(\hat {\xi})+O(1)$ where 
$\hat {\xi}=\xi-\hat {\alpha}$.  Then, 
$f\vert \mathcal O_\tau= -\log \vert \hat {s}_0\vert_h^2-\delta$
where 
$$
\delta =\lim_{N\to \infty}\frac{1}{N}\log \vert \chi_{\alpha_N}\vert^2\,.
$$
In order to interpret this correctly, we must identify $\mathcal O_\tau$
with the subgroup $T'\cong (\mathbb C^*)^k$, use the 
$T'_{\mathbb R}$-invariance of $\vert \hat {s}_0\vert_h^2$ and
express these quantities in polar coordinates as in Proposition 
\ref{convex exhaustion}. At that level, i.e., in logarithmic 
coordinates, we see that $f$ is the translate by the linear function 
$2\pi \mathrm {Im}(-\hat {\xi})$ of the strictly convex exhaustion of 
$\mathbb R^k$ determined by $-\log \vert \hat{s}_0\vert_h^2$. The influence 
of such a linear translation on the minimizing set of a 
strictly convex function is easily understood. It also should be noted
that $M$ is just the preimage of $0$ under the moment map defined
by the K\"ahlerian potential $f$ on $X\setminus Y$ (see, e.g.,\cite{HH1}).
\subsection{Localization}
In order to obtain estimates of integrals involving the functions
$f_N$ we will use information on their restrictions to orbits of
1-parameter subgroups which close up to points on the closed orbit
$T.x_0^\tau$. For this we use the following fact.

\begin {lemma}\label{closing up}
If $\lambda :\mathbb C^*\to T$ is a 1-parameter subgroup and $x\in
X\setminus Y$ such that the orbit $\lambda (\mathbb C^*).x$ is not
closed, then the closure $X\setminus Y$ consists of the orbit plus
one additional point $b$ with $f_N(b)$ strictly less than $f_N(x)$
for all $N$.
\end {lemma}
\begin {proof}
Since $X\setminus Y$ is affine, the closure $C$ of such an orbit is
not compact and therefore contains exactly one additional point. Of
course $C$ may be singular, but it is locally irreducible so that
the normalization $\widehat {C}\to C$ is injective and the pullback
of $f_N$ to $\widehat {C}\cong \mathbb C$ is an $S^1$-invariant
plurisubharmonic exhaustion exhaustion which is strictly
plurisubharmonic outside of the preimage $\widehat b$ of $b$. If
strictly inequality would not hold on some circle $S^1.z$ for $z\in
\widehat {C}\setminus \{\widehat b\}$, then the maximum principle
would be violated.
\end {proof}
After these preparations we are now in a position to prove the
desired estimates.
\begin {proposition}[Localization Lemma]
If $U(M)$ is a $T_{\mathbb R}$-invariant neighborhood of $M$ which
is relatively compact in $X\setminus Y$, then there exists $N_0$ and
$\varepsilon >0$ so that if $N>N_0$ it follows that $f_N>\varepsilon
$ on the complement of $U(M)$ in $X\setminus Y$.
\end {proposition}
\begin {proof}
Choose $N_0$ so that $N>N_0$ implies that the minimum set $M_N$ is
contained in $U(M)$. Since $f_N$ converges uniformly to the strictly
plurisubharmonic function $f$ on $U(M)$ we may also assume that
$f_N>\varepsilon >0$ on the boundary of $U(M)$. Given $x\in
X\setminus Y$ the Hilbert Lemma guarantees the existence of a
1-parameter group $\lambda :\mathbb C^*\to T$ whose orbit $\lambda
(\mathbb C^*).x$ closes up to the closed $T$-orbit $T.x_0^\tau$. By Lemma
\ref{closing up}, if $x\not\in T.x_0^\tau$, then $f_N(x)>f_N(b)$ where
$b$ is the additional point in the closure. If $b\not\in U(M)$, then
by connecting $b$ to $M_N$ by a real 1-parameter group and using the
strong convexity of $f_N$ along that orbit, we see that $f_N(b)$ is
larger than the value of $f_N$ at the intersection of that orbit
with the boundary of $U(M)$. Thus, unless $x$ is already in $U(M)$,
it follows that $f_N(x)>\varepsilon $.
\end {proof}
We refer to the above result as a \emph{Localization Lemma}, because
it implies that integrals localize at $M$. Here is an example of
what we mean by this.
\begin {corollary}\label{labThmRapidDecrease}
For $U(M)$, $\varepsilon $ and $N_0$ as above, given $h\in L_1(X)$
it follows that $$ \int he^{-Nf_N}d\lambda \le \Vert
h\Vert_{L_1}e^{-N\varepsilon}+ \int_{U(M)}he^{-Nf_N}d\lambda \,. $$
\end {corollary}
Below we prove precise estimates which lead to the desired result
that the probability function $\vert \varphi_N\vert^2$ converges to
the Dirac measure of $M$.

\subsection {Morse property}\label{Morse}

Here we show that $f$ is a Bott-Morse function near $M$.
For this the essential point is to understand the behavior of the
extension of $f$ to a smooth embedding space of $\Sigma_\tau $.
\subsubsection* {Extending from $\Sigma_\tau $}
Recall that $X\setminus Y$ is a naturally identifiable with the
product $\Sigma_\tau \times \mathcal O_\tau$ with $M$ contained in
$\mathcal O_\tau $ as the $T_{\mathbb R}$-orbit $T_{\mathbb R}.x_1$.
We begin by analyzing the local behavior of $f$ on $\Sigma_\tau $. 
For this we first $T_\tau$-equivariantly embed $\Sigma_\tau $ in
a complex vector space $W$ where $x_0^\tau$ is mapped to the origin
$0\in W$.

\bigskip\noindent
Now recall that since $x_0^\tau$ is the $T_\tau$-fixed point in
$\Sigma_\tau $, there exists a 1-parameter subgroup $\lambda $ of
$T_\tau $ with the property that $\lim_{t\to 0}\lambda (t).x=x_0^\tau$
for every $x\in \Sigma_\tau$. We refer to $x_0^\tau$ as being
\emph{attractive} for $\lambda $. Any linearization such as $W$
splits $W=W_-\oplus W_0\oplus W_+$ with respect to any such
1-parameter subgroup where $W_-$ are the points in $W$ for which $0$
is attractive for $\lambda $, $W_0$ is the set of $\lambda $-fixed
points and $W_+$ is the set of points for which $0$ is repulsive. In
our case every point of $\Sigma_\tau $ is in $W_-$. So we have the
following remark.

\begin {proposition}\label{attractive}
There exists an equivariant embedding $\Sigma_\tau \hookrightarrow
W$ so that $0$ is attractive for $\lambda $ for every point in $W$.
In particular, if in its linearization $\lambda (t):=\mathrm
{Diag}(\chi_1(t),\ldots ,\chi_n(t))$, then the weights defining 
the $\chi_j$ are all negative, e.g., it is never the case that
$\chi_i\chi_j=1$.
\end {proposition}

Since $f$ is a smooth strictly plurisubharmonic function, it extends
to a neighborhood $U$ of $0$ in $W$ as a smooth strictly
plurisubharmonic function which after averaging is invariant with
respect to the compact form $(T_\tau )_{\mathbb R}$. We may
assume that $U$ is $(T_\tau )_{\mathbb R}$-invariant and that 
$\lambda (t)(U)\subset U$ as $t\to 0$. For $x\in U$ we consider the closure
$\mathrm{cl}(\lambda (\mathbb {C}^*).x$ and let $C(x)$ be its
intersection with $U$.

\begin {proposition}
The origin $0\in U$ is the absolute minimum point of $f$ in $U$,
i.e., if $0\not=x\in U$, then $0=f(0)<f(x)$.
\end {proposition}

\begin {proof}
The normalization $\widehat {C}(x)$ of $C(x)$ may be identified with
the unit disk so that the pullback $\widehat {f}$ of $f$ is
plurisubharmonic and strictly plurisubharmonic outside of the
origin. Since $\widehat f$ is $S^1$-invariant, the desired result
follows immediately from the meanvalue property and the maximum
principle.
\end {proof}

Now let us consider the Taylor development of $f$ at $0$,
\[
  f(z)=Q(z)+O(3)\,. 
\]
The second order terms are of the form $Q(z)=\bar {z}^THz+R(z)$
where $H$ is a Hermitian matrix and $R(z)=\mathrm {Re}(\sum
a_{ij}z_iz_j)$. We have implicitly chosen the coordinates $z$ where
$\lambda $ is linearized and therefore with the origin being
attractive for $\lambda $.

\begin {proposition}
In the above coordinates $R(z)\equiv 0$.
\end {proposition}

\begin {proof}
The function $f$ is invariant with respect to the $S^1$-action
defined by $\lambda $. Therefore $R$ is invariant as well. On the
other hand $R(t(z))= \mathrm {Re}(\sum
a_{ij}\chi_i(t)\chi_j(t)z_iz_j)$ for $t\in S^1$. Hence the
desired result follows from Proposition \ref{attractive}.
\end {proof}

\begin {corollary}
The extended strictly plurisubharmonic funtion $f$ is a
$(T_\tau)_{\mathbb R}$-invariant Morse function with absolute minimum at $0$.
\end {corollary}

\subsubsection* {Morse property on the full neighborhood of $M$}

Having shown that $f\vert \Sigma_\tau$ can be regarded as the
restriction of a $(T_\tau )_{\mathbb R}$-norm function we now
deal with the restriction of $f$ to the orbit $\mathcal
O_\tau$. Let $T'$ be a toral subgroup complementary to $T_\tau$
which acts freely and transitively on $\mathcal O_\tau $. Using
polar coordinates we regard the quotient $\mathcal
O_\tau/T'_{\mathbb R}$ as a vector space $V$ with the image of $M$
being the origin. The function $f\vert \mathcal O_\tau$ is the
pullback of a strictly convex function on $V$ which attains its
minimum as a nondegenerate critical point at the origin.

\bigskip\noindent
Now recall that the product structure $X\setminus Y=\Sigma_\tau
\times \mathcal O_\tau$ is defined by the $T'$-action as $X\setminus
Y=\Sigma_\tau \times T'$. In this way we regard its quotient by
$T'_\mathbb R$ by as the product $\Sigma_\tau \times V$ and we embed
this in $W\times V$ where $\Sigma_\tau \hookrightarrow W$ is
embedded as in the previous section. We regard $f$ as being defined
on this quotient. For the following result recall that $U$ is the
neighborhood of $0$ to which $f\vert \Sigma $ extends as a Morse
function with $0$ its absolute minimum as a nondegenerate critical
point.

\begin {proposition}\label{labThmMorseProperty}
The extended function $f$ on $U\times V$ is a Morse function with
its only critical point being the origin which is its absolute
minimum.
\end {proposition}
\section {Proofs of the main results}\label{labSecApplications}
Here we apply the results of \S{}\ref{labSecStrictlyPSlimit} to prove
the theorems which are announced in $\S\ref{labSecTorusLifting}$. We
begin by commenting on the difference between the estimates for an
an arbitrary sequence approximating a ray and a tame 
sequence approximating the same ray.
\subsection {Arbitrary sequences}\label{arbitrary sequences}
Recall that we replace a given sequence $(s_N)$ by a \emph{tame}
sequence $(s_N')$ with $f_N'$ converging to a strictly
plurisubharmonic function $f'$ on a certain Zariski open set
$X\setminus Y$ which is canonically defined by the ray $R(\xi)$.
The essential results for $(s_N')$ are proved in the following
paragraphs of this section.  In particular it is shown
that the probability density
$\vert s_N'\vert / \Vert s_N'\Vert^2_{L_2}$ converges in measure
(with precise estimates) to the Dirac measure of a canonically
associated $T_{\mathbb R}$-orbit $M$. Here $M$ is the set where $f'$
takes on its minimum.

\bigskip\noindent
To complete the project we return to our considerations of the
original sequence. Let $(s_N)$ be a given sequence which
approximates the ray $R(\xi)$ at infinity and let $(s_N')$ an
associated tame sequence. It follows that there exists a bounded
sequence of linear functions $\beta _N\in \mathfrak t^*$ so that
\[
  \Bigl| \frac{s_N}{s_N'} \Bigr|^2 = 
    e^{2\mathrm{Re}(\beta _N)}\,.
\]
In other words there are characters $\chi_N$ belonging to a finite
set so that 
\[
  \vert s_N\vert^2 = \vert \chi_N\vert^2\vert s_N'\vert^2
\]
on the open orbit. In the example at the beginning of 
$\S\ref{labSecStrictlyPSlimit}$, using the coordinate $z=z_1z_0^{-1}$
the coefficient $\vert \chi_N\vert^2$ is just $\vert z\vert^2$.

\bigskip\noindent
Thus we must compute the integral
\[
  \int _X \vert \chi_N\vert^2e^{-Nf'}\,.
\]
The key is that the characters $\chi_N$ which arise here belong to a
finite set. Furthermore, except on arbitrarily small neighborhoods
of the minimizing orbit $M$ of $f'$, the term $e^{-Nf'}$ kills the
effect of the these coefficients. Finally, the 
$\chi_N$ extend to holomorphic functions $m_N$ on $X\setminus Y$.
These have a certain vanishing order which contributes to
the integral $\Vert s_N\Vert^2$ just as in the case of the example
of $\mathbb P_1$. In order to show that the probability densities
converge to the Dirac measure on $M$ it is enough to show that they
do so for a subsequence where the coefficient characters are
constant with vanishing order $d$ along a divisor containing $M$. In
this case, when computing $\Vert s_N\Vert^2$ we have a
correction term of order $N^{-d}$. However, this only has an effect
on the speed of convergence to $\delta_M$. Of course the function
$D_N(t):=\mathrm {Vol}\{\vert \varphi_N\vert_h ^2>t\}$ will be
affected, but the upper estimate for this will be given by the tame
sequence. Having made these remarks, for the remainder of the
paper we consider only tame sequences.

\subsection{Pointwise asymptotics of probability densities}

Using the global product structure of $X\setminus Y$ we 
now introduce a local basis of $T$-invariant neigbhorhoods $U(M)$ of $M$
in $X\setminus Y$ which is appropriate for our purposes. 
For this let $V(M)$ be an arbitrarily small
$T$-invariant neighborhood of $M$ in $\mathcal O_\tau $ and $\Delta $ 
an arbitrarily small $T_\tau $-inariant neighborhood of the fixed
point $x_0^\tau\in \Sigma_\tau$. Using the map $\alpha $ we regard
$U(M):=V(M)\times \Delta $ as an arbitrarily small $T$-invariant
neighborhood of $M$ in $X\setminus Y$. Recall
that we are only dealing with a tame sequence with defines
the strictly plurisubharmonic function $f$ which 
takes on its minimum exactly on $M$ and that
this minimum value has been normalized to be zero.

\bigskip\noindent
Since $f_N\to f$ uniformly on $U(M)$, it is a simple matter
to make pointwise estimates of $\vert s_N\vert^2=e^{-Nf_N}$.
This is due to the fact that $f$ is a smooth strictly plurisubharmonic
function which takes on its minimum along $M$ and is a Bott-Morse function
there. In the orbit direction
of $V(M)$ we can choose coordinates so that it is realized as the
pullback of a Morse function from $\mathcal O_\tau$ and in the
slice $\Delta $ we can write it as a Morse function at the fixed point.
Combining estimates using these quadratic forms, we will determine
precise estimates for $D_N(t)$.

\bigskip\noindent
In order to provide the necessary estimates for the probability
density function we must approximate the $L_2$-norm $\Vert
s_N\Vert^2$. Using localization and the fact that $f_N\to f$
uniformly on compact subsets of $X\setminus Y$ it is sufficient to
obtain estimates for
\[
  I_N = \int_{U(M)}e^{-Nf}
\]
where $U(M)=V(M)\times \Delta $ is an arbitrarily small product
neighborhood of $M$.
\begin {proposition}\label {integral estimate}
There exists a positive constant $c$ such that
\[
  I_N \sim cN^{-\kappa }
\]
where $\kappa $ is the sum of the complex dimension of $\Sigma_\tau$
and one-half the complex dimension of $\mathcal O_\tau$.
\end {proposition}
The notation here means that $\lim_{N\to \infty}N^{\kappa}I_N=c$.
It should be remarked that by using the 
asymptotic development in $\S\ref{appendix}$ 
and an elementary asymptotic development of the integral along $V(M)$, 
it would be possible to give a more complete asymptotic development of $I_N$.

\bigskip\noindent
{\it Proof of Proposition \ref{integral estimate}.} We write $I_N$
as a double integral
\[
  I_N = \int_{x\in V(M)}e^{-Nf(x)}
    \int_{\Delta_x}e^{f-f(x)}
\]
where $\Delta_x$ is the fiber over $x$ of the production $V(M)\times
\Delta \to V(M)$ and $f(x)$ denotes the value of $f$ at the point in
$\Delta_x$ where it takes on its minimum, i.e., at the intersection
of that fiber with $\mathcal O_\tau$. Of course we view $x\in
V(M)\subset \mathcal O_\tau$. For each fiber $\Delta _x$ we may
apply Corollary \ref{key estimate} which implies that there is a
positive continuous function $c=c(x)$ so that $$ \int_{\Delta
_x}e^{f-f(x)}\sim c(x)N^{-d} $$ where $d$ is the complex dimension
of $\Delta_x$. Thus it remains to compute
\begin {equation}\label{orbit integral}
\int_{x\in V(M)}c(x)e^{-Nf(x)}\,.
\end {equation}
As noted above we may assume that $V(M)$ is the product of the torus
$T'$ and a ball $B$ so that $f=f(x)$ is the lift from the ball of a
positive definite quadratic form. The orbit $M$ is the preimage of
the origin in the projection $T\times B\to B$. Explicit computation
of an integral of the form
\[
  \int_{y\in B}e^{-N\Vert y\Vert^2}
\]
shows that the integral in (\ref{orbit integral}) is just a constant
times $N^{-\frac{d}{2}}$ where $d$ is the complex dimension of
$V(M)$.\qed

\begin{corollary}\label{labThmProbDensAsymp}
For a tame sequence $\{s_N\}$ the pointwise asymptotic behavior of
the associated probability functions $|\varphi_N|_h^2$ is given by
\begin{equation}
  |\varphi_N|_h^2 \sim N^\kappa e^{-Nf}
\end{equation}
where $\kappa = \dim \Sigma_\tau + \frac{1}{2}\dim \mathcal O_\tau$.
\end{corollary}

\subsection {Distribution functions}\label{labSecDistFunc}

In $\S\ref{Morse}$ we showed hat $f$ can be regarded
as a smooth Morse function in the appropriate embedding space for
$\Sigma_\tau$. Using this, the estimate for the volume $D_N(t)$ in the
tails of the distribution follows from known estimates for the
volume of an analytic set embedded in a ball in $\mathbb C^n$.
\subsubsection*{The volume estimate}
Recall that we are interested in computing
\[
  D_N(t) := \mathrm {Vol}\{x\in X; \vert \varphi_N\vert^2_h>t\}\,. 
\]
Since we have localized the integral to an arbitrarily small
neighborhood of $M$ and since the free $T'_{\mathbb R}$-action
leaves all relevant quantities invariant, it is enough to compute
$D_N(t)$ in the local $T'$-quotient of $X$ in $U\times V$ where the
quotient of $X$ is realized as $(\Sigma_\tau \cap U)\times V$. Now
we choose coordinates for $U$ and $V$ so that the extended function
$f$ is a sum of norm functions: $f=\Vert \ \Vert_U^2+\Vert \
\Vert_V^2=:\Vert \ \Vert^2$. For the computation of $D_N(t)$ it is
then enough to compute $\mathrm {Vol}\{x\in (\Sigma_\tau \cap
U)\times V=:A; \Vert x\Vert^2<r \}$. In other words we consider the
ball $\mathbb {B}(r)$ and compute the volume of the (pure
dimensional) analytic subset $A\cap \mathbb B(r)$. It is known
that this is $c(r)r^{d}$ where $c(r)$ is a continuous function
bounded from above and below, which is closely related to the degree
of $A$, and $d$ is the (real) dimension of $A$.
In other words this volume is asymptotically the same as the volume
of a linear submanifold of the same dimension. This allows us to
immediately compute $D_N(t)$

\begin {proposition}\label{labThmDistFuncTame}
The unscaled volume for a tame sequence $\{s_N\}$ is given by 
\[
  D_N(t)\sim \Big(\frac{\log N}{N}\Big)^\kappa\,. 
\]
\end {proposition}

\begin {proof}
For $N>\!\!>0$ we may replace $f_N$ by $f$ and $\vert
\varphi_N\vert^2_h$ by $N^\kappa e^{-Nf}$. We then note that
$N^\kappa e^{-Nf}>t$ is the same as
\[
  f< \frac{\kappa \log N}{N} - \frac{\log t}{N}t 
    \sim \frac{\log N}{N} =: r 
\]
where we compute $\mathrm {Vol}(\{f<r\})$ as above.
\end {proof}

Combining this with the remarks of \S{}\ref{arbitrary sequences} we
obtain Theorem \ref{labThmDistFunc}.

\begin{remark}
1. If $(s_N)$ is not tame, then (localized near $M$)
the norm $\vert s_N\vert^2_h$ may become smaller than that of an
associated tame sequence. Thus the above precise asymptotic
expression for $D_N(t)$ would become an estimate from above. 2. In
\cite {STZ} asymptotic expressions for $D_N(t)$ are obtained in certain
situations (see Theorem 1.3 of that paper).
\end{remark}

\subsection{Convergence of measures}

Using arguments to those in \S{}\ref{labSecDistFunc} we now prove
that the probability densities $|\varphi_N|_h^2$ converge weakly to
the Dirac measure on $M$ (Theorem \ref{labThmConvDirac}). In other
words, for any continuous function $u$ 
\[
  \int_X u \, |\varphi_N|_h^2 d\lambda \rightarrow 
    \int_M u \, dM
\]
It is well-known (see \cite{Hoermander}, Chapter II) that it
suffices to show this for a smooth, compactly supported function
$u$.

\bigskip\noindent
Recall that $U(M)$ is a $T_\mathbb R$-invariant neighborhood of $M$
on which the functions $-\frac{1}{N} \log|s_N|_h^2$ converge
uniformly to the strictly plurisubharmonic limit function $f$
discussed in \S{}\ref{labSecStrictlyPSlimit}. Outside this
neighborhood the functions $|\varphi_N|_h^2 = |s_N|_h^2 /
\|s_N\|_{L^2}^2$ are rapidly converging to zero (see Corollary
\ref{labThmRapidDecrease}). Combining this with the asymptotic
formula for the probability densities (see Corollary
\ref{labThmProbDensAsymp}) we obtain
\[
  \int_X u \, |\varphi_N|_h^2 d\lambda
    = N^d \int_{U(M)} u \, e^{-Nf} d\lambda 
       	+ O(e^{-\varepsilon N}) .
\]
Recall that for a tame sequence the exponent $d$ in the above
equation is precisely given by $\kappa$ in Corollary
\ref{labThmProbDensAsymp}. For a non-tame sequence it might be
bigger since the vanishing order of the $s_N$ increase (see remarks
in \S{}\ref{arbitrary sequences}). However, this only improves
the convergence of the measures.

\bigskip\noindent
We continue by using the Morse property of $f$ as explained in 
$\S\ref{labSecDistFunc}$. For this we write $U(M) = T'_\mathbb R \times
U \times V$, where $U \times V$ is the product neighborhood provided
by Proposition \ref{labThmMorseProperty}. Expressing $u$ in the
respective coordinates we obtain
\[
  \int_{U(M)} u \, e^{-Nf} d\lambda 
    = C \int_{T_\mathbb R'} 
	\int_{U \times V} u(\vartheta, z, \mu) \,
	  e^{-N (\|z\|^2 + \|\mu\|^2)} d\lambda \, d\vartheta
\]
Finally, using the Taylor expansion of $u(\vartheta, \cdot)$ we
see that the right-hand side of the above equation converges to
\[
  \int_{T'_\mathbb R} u(\vartheta, 0, 0) d\vartheta 
    = \delta_M(u) .
\]
This completes the proof of Theorem \ref{labThmConvDirac}. 

\section {Appendix}\label{appendix}
The aim of this Appendix is to prove the following theorem :
\begin{theorem}\label{asympt.1}
Let $X \subset U$ be an analytic set in open neighbourhood $U$ 
of the origin in  $\mathbb{R}^m$ which is assumed to be of pure 
dimension $n+1   $. Let $f : U \to \mathbb{R}^{+} $ be a real 
analytic function which is strictly positive on $X$  outside 
$\{0\}$, and let $\omega$ be a smooth compactly support 
$(n+1)-$form on $U$. Define the function 
$F : \mathbb{R}^+ \to \mathbb{R}$ by putting  
for $t \in \mathbb{R}^{+}$ 
$$ F(t) : = \int_X \ e^{-t.f}.\omega .$$
Then as $t\to \infty$ the function $F$ admits 
an infinitely differentiable asymptotic expansion of the following form:
\begin{equation*}
 F(t) \simeq \sum_{i \in [1,p]} \sum_{j \in [1,n]} \sum_{\nu \in \mathbb{N}^*} \ c_{i,\nu}^j.t^{-(\alpha_i+\nu)}.(\operatorname{Log} t)^j  \tag{1}
 \end{equation*}
where $\alpha_1, \dots, \alpha_p$ are rational 
numbers strictly bigger than $-1$. \\
Moreover, the rational numbers $\alpha_1, \dots, \alpha_p$ and 
the exponents of $\operatorname{Log} t$ have to be present in the 
asymptotic expansion at $s \to 0^+$ for the fiber integral
$$ \varphi(s) : = \int_{X \cap \{f=s\}} \ \omega\big/df .$$
\end{theorem}
For the proof we begin with the remark that the Fubini theorem gives
\begin{equation*}
F(t) = \int_0^{+\infty} \  e^{-t.s} \varphi(s).ds \tag{2}
\end{equation*}
where $\varphi$ is the fiber-integral of the statement. 
Now it known\footnote{See \cite{J.70}, \cite {AVG}  or \cite{B.04} 
for more information on the exponents in the case of an isolated 
singularity.} that $\varphi$ admits an  infinitely differentiable 
asymptotic expansion when $s \to 0^+$ of the form
\begin{equation*}
\varphi(s) \simeq \sum_{i \in [1,p]} \sum_{j \in [1,n]} 
\sum_{\nu \in \mathbb{N}} \ \gamma_{i,\nu}^j \ s^{\alpha_i+\nu}.(\operatorname{Log} s)^j \tag{3}
\end{equation*}
so the proof of the theorem is a consequence of the following  
proposition $\qed $
\begin{proposition}\label{Laplace}
Let $\varphi :\, ]0, +\infty[ \to \mathbb{R}$ be a continuous 
function with support in $]0,A]$,  which admits, when $s \to 0^+$, 
an asymptotic expansion of the type $(3)$. Then the function defined 
in $(2)$  admits when $t \to + \infty$ an infinitely 
differentiable asymptotic expansion of the form  $(1)$. Moreover, 
the rational numbers $\alpha_i$  which appear in the 
expansion of $F$ have to appear in the expansion $(3)$. 
For each $i \in [1,p]$, the maximal power of $\operatorname{Log}$ 
which appears in $(1)$ with some $s^{\alpha_i+\nu}$ 
in front is bounded by its analog in $(3)$.
\end{proposition}
The proof of this proposition is contained in the following elementary results.
\begin{lemma} For any $\alpha > -1$ and any 
$j \in \mathbb{N}$ there is a monic polynomial $P_j$ of 
degree $j$ such that 
$$ \int_0^A e^{-t.s}.s^{\alpha}.(\operatorname{Log} s)^jds \simeq 
 \frac{(-1)^j}{\Gamma(\alpha+1)}.P_j(\operatorname{Log} t).\frac{1}{t^{\alpha+1}} + 0(e^{-\varepsilon.t})$$
for any given $\varepsilon > 0$  small enough and for $t \gg 1$.
\end{lemma}
\begin {proof} Let us first show that
$$ \int_0^{+\infty} \ e^{-t.s}.s^{\alpha}.(\operatorname{Log} s)^jds = 
\frac{(-1)^j}{\Gamma(\alpha+1)}.P_j(\operatorname{Log} t).\frac{1}{t^{\alpha+1}} $$
where $P_j$ is a monic polynomial of degree $j$. 
This formula is easy for $j = 0$. The general case is obtained by 
$j$ derivations in $\alpha$.\\
Now, as the function $e^{-u/2}.u^{\alpha}.(\operatorname{Log} u)^j$ 
is decreasing for $u \gg1$, it follows that for 
$t \gg 1$ the function $\int_{t.A}^{+\infty} \ e^{-u}.u^{\alpha}.(\operatorname{Log} u)^j$
 is $0(e^{-\varepsilon.t})$ for any given $\varepsilon >0$ 
small enough. 
\end {proof}

{\bf Remark.} The preceding proof gives a more precise relation between the expansions $(3)$ and $(1)$. \\

Now the following elementary lemma allows one to cut asymptotic expansions.
\begin{lemma}
In the same situation as in the previous lemma assume that \\
 $\vert \varphi(s)\vert \leq C.s^N$ for $s \in ]0,A]$. Then 
$$\vert \int_0^A \ e^{-t.s}.\varphi(s).ds \vert \leq C.\frac{(N+1)!}{t^{N+1}}.$$
\end{lemma}
{\bf Remark.} Assume that in the situation of the 
theorem the function $f$ is no longer real analytic, but 
still continuous and with an isolated zero on $X$ at the origin. 
If we have two analytic functions $g_-$ and $g_+$ such 
that they satisfy the hypothesis of the theorem and the inequalities :
$$ g_- \leq f \leq g_+ $$ 
near $0$ in $X$, then we may deduce some estimation for 
$F(t)$ when $t \to + \infty$.\\
For instance if $f$ is a positive Morse function of class $\mathcal{C}^3$ 
near $0$ in $\mathbb{R}^m$ such that $f(0) = 0$ and 
$df_0 = 0$, we may use $g_-(x) = (1- \varepsilon).q(x)$ and 
$g_+(x) = (1+\varepsilon).q(x)$ where $q : = d^2f_0$ is a 
non degenerate positive quadratic form to obtain, at least, 
the first term of an expansion of $F(t)$ when $t \to +\infty$.\\
In the case $X$ is a complex analytic subset of dimension pure 
$n$ near the origin in $\mathbb {C}^N$ we obtain the following 
corollary.
\begin{corollary}\label{key estimate}
Let $X \subset U$ be a complex analytic subset of pure dimension $n$ 
in an open neighbourhood $U$ of the origin in $\mathbb {C}^N$. 
Assume that $0 \in X$ and that $f : U \to \mathbb{R}^+$ is a 
$\mathcal{C}^3$ Morse function on $U$ with only one critical 
point at $0$. Let $\omega$ be any $\mathcal{C}^{\infty}$ 
K{\"a}hler form on $U$ and $\rho : U \to \mathbb{R}^+$ be a 
$\mathcal {C}^0$ function with compact support which is equal to 
$1$ in a neighbourhood of the origin. Then the limit
$$ \lim_{t \to +\infty} \ t^n.\int_X e^{-t.f}.\rho.\omega^{\wedge n} $$
exists and is finite and strictly positive.
\end{corollary}
\begin {proof} First we remark that the choice of $\rho$ 
is irrelevant, because the change of $\rho$ will add a 
$0(e^{-\varepsilon.t})$ for some $\varepsilon > 0$. The 
observation made  before this corollary now implies that for any 
$\varepsilon > 0$, the inequalities
$$ \int_X \ e^{-t.g_+}.\rho.\omega^{\wedge n} \leq 
\int_X e^{-t.f}.\rho.\omega^{\wedge n} \leq  
\int_X \ e^{-t.g_-}.\rho.\omega^{\wedge n} $$
and if the result is proved for $f = q$, we deduce the result for $f$.\\
Now we may use the local parametrization theorem for the complex 
analytic set $X$ near zero and see that if 
$X \subset P \times B $ where $P$ is a polydisc in $\mathbb {C}^n$ 
such the projection on $P$ gives a proper and finite map 
$\pi : X \to P$, which is also proper and finite on the Zariski 
tangent cone of $X$ at the origin, the result is true for 
$f(u,x) = \vert\vert u \vert\vert^2 + \vert\vert x\vert\vert^2$,  
$\omega$ a K{\"a}hler form on $P$  and $\rho = \pi^*(\sigma)$ 
with $\sigma$ a $\mathcal{C}^0$ function with compact support in 
$P$  which is equal to $1$ in a neighbourhood of the origin. \\
In this case, the inequality 
$ \vert\vert x\vert\vert \leq C.\vert\vert u\vert\vert $ 
which is true for $(u,x) \in X$ near enough $(0,0)$,  
thanks to the tranversality of $\{0\}\times B$ to $C_{X,0}$, gives
$$k. \int_P e^{-t.(1+C^2).\vert\vert u\vert\vert^2}.\sigma.\omega^{\wedge n} 
\leq \int_X e^{-t.f}.\rho.\omega^{\wedge n} \leq k.
\int_P e^{-t.\vert\vert u\vert\vert^2}.\sigma.\omega^{\wedge n} $$
where $k$ is the degree of $X$ on $P$.\\
Finally, we reach the case where $\omega$ is a K{\"a}hler form on 
$U = P \times B$ using the fact that we may find a finite number of 
projections such that the given K{\"a}hler form is bounded by the 
sum of the pull back of the K{\"a}hler form on $P$ by these projections. 
Because we know that the asymptotic expansion exists, this again shows 
that there exists a non-zero limit when we multiply by a suitable factor 
$t^{\alpha}\big/(\operatorname{Log} t)^j$ and that we have $\alpha = n$ and 
$j = 0$ as in the case where 
$X = P, f = \vert\vert u \vert\vert^2$ 
and $\omega^{\wedge n} = (-2i)^n.du\wedge d\bar u$. 
\end {proof}
%
%
%
%

\end {document}